\documentclass{amsart}
\usepackage[utf8]{inputenc}

\usepackage{graphics}
\usepackage{thmtools}
\usepackage{wasysym}
\usepackage[T1]{fontenc}    

\usepackage{amsthm}
\usepackage{amsbsy,amsmath,amssymb,amscd,amsfonts}
\usepackage[pagebackref=true]{hyperref}

\usepackage{graphicx,float,latexsym,color}
\usepackage[font={scriptsize,it}]{caption}
\usepackage{subcaption}

\usepackage{makecell}
\renewcommand{\arraystretch}{1.2}

\usepackage[dvipsnames]{xcolor}

\newtheorem{theorem}{Theorem}
\newtheorem*{theorem*}{Theorem}

\newtheorem{proposition}{Proposition}

\newtheorem{corollary}{Corollary}
\newtheorem{lemma}{Lemma}
\theoremstyle{remark}
\newtheorem{remark}{Remark}
\theoremstyle{definition}

\hypersetup{
    pdftoolbar=true,        
    pdfmenubar=true,        
    pdffitwindow=false,     
    pdfstartview={FitH},    
    colorlinks=true,       
    linkcolor=OliveGreen,          
    citecolor=blue,        
    filecolor=black,      
    urlcolor=red           
}

\usepackage{lineno}

\usepackage{comment}
\usepackage{microtype}
\usepackage{footnote}

\title[Related by Similarity II: Homothetic Pair and Brocard Porism]{Related by Similarity II: Poncelet 3-Periodics in\\the Homothetic Pair and the Brocard Porism}
\author{Dan Reznik} 
\author{Ronaldo Garcia} 
\date{September, 2020}

\begin{document}

\maketitle

\begin{abstract}
Previously we showed the family of 3-periodics in the elliptic billiard (confocal pair)  is the image under a variable similarity transform of poristic triangles (those with non-concentric, fixed incircle and circumcircle). Both families conserve the ratio of inradius to circumradius and therefore also the sum of cosines. This is consisten with the fact that a similarity preserves angles. Here we study two new Poncelet 3-periodic families also tied to each other via a variable similarity: (i) a first one interscribed in a pair of concentric, homothetic ellipses, and (ii) a second non-concentric one known as the Brocard porism: fixed circumcircle and Brocard inellipse. The Brocard points of this family are stationary at the foci of the inellipse. A key common invariant is the Brocard angle, and therefore the sum of cotangents. This raises an interesting question: given a non-concentric Poncelet family (limited or not to the outer conic being a circle), can a similar doppelgänger always be found interscribed in a concentric, axis-aligned ellipse and/or conic pair?

\end{abstract}

\section{Introduction}
\label{sec:intro}
Previously we studied invariants in the so-called Poristic triangle family (fixed incircle and circumcircle) \cite{gallatly1914-geometry} in relation to 3-periodics in the elliptic billiard \cite{garcia2020-poristic}. We found both these families were images of one another under a variable similarity transform. By definition, poristic triangles conserve inradius $r$ and circumradius $R$; we showed billiard triangles conserve $r/R$ \cite{garcia2020-new-properties,reznik2020-intelligencer}. Since $r/R=1+\sum{\cos\theta_i}$, both conserve the sum of cosines.

Here we study a new duo of Poncelet 3-periodics families which are also related by a similarity and also share a common invariant (the sum of cotangents). These arise as follows:

\begin{enumerate}
\item The homothetic pair: an external ellipse with semi-axes $(a,b)$, and an internal, concentric, axis-aligned one with semi-axes $(a',b')=(a/2,b/2)$. These are known as the Steiner circum- and inellipse, respectively. 
\item Brocard Porism: a fixed circumcircle and a caustic known as the Brocard Inellipse \cite[Brocard Inellipse]{mw}.
\end{enumerate}

While the first one preserves area and its barycenter is stationary (it is an affine image of the concentric circular pair) the second one conserves the Brocard angle $\omega$; see Figure~\ref{fig:brocard-basic}. This family is remarkable as its Brocard points \cite[Brocard Points]{mw} are stationary at the inellipse foci.

\begin{figure}
    \centering
    \includegraphics[width=.7\textwidth]{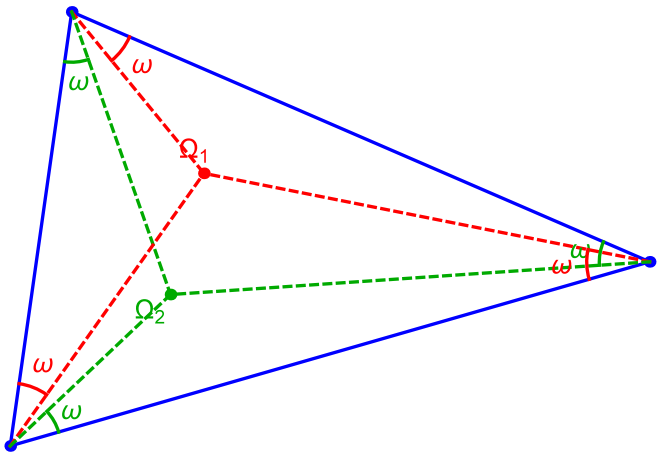}
    \caption{For every triangle two Brocard Points exist where sides ${P_i}{P_{i+1}}$ concur when rotated an angle $\omega$ about $P_i$ toward the triangle's interior. If vertices are traversed counterclockwise (resp. clockwise) one obtains $\Omega_1$ (resp. $\Omega_2$).
    }
    \label{fig:brocard-basic}
\end{figure}

\subsection*{Main results} we (i) prove the homothetic family conserves $\omega$. We then (ii) describe a 2d parametrization for the family of triangles with fixed Brocard points in termos of the axes of the Brocard inellipse. Finally we show that (iii) the homothetic and Brocard-Poristic families are similar. Since this transform is angle-preserving, it is consistent with the fact that both families conserve Brocard angle. In turn this implies both conserve the sum of cotangents, since $\cot\omega=\sum{\cot\theta_i}$ \cite[Brocard Angle, Eq. 1]{mw}.

\subsection*{Related Work}

Properties and relations involving the Brocard points have been widely studied \cite{coolidge1971,gallatly1914-geometry,honsberger1995,shail1996-brocard}.

Johnson terms {\em equibrocardal} families which conserve $\omega$ and proves that the projection of a family of equilaterals embedded in a tilted plane onto the horizontal plane is equibrocardal \cite[Chapter XVII]{johnson1960}. Pamfilos describes equibrocardal isosceles triangles arising from projections on the Lemoine axis  \cite{pamfilos2004}. 

Bradley studies the locus of triangle centers over the Brocard porism in \cite{bradley2007-brocard} and defines conics associated with the Brocard points and porism in \cite{bradley2011-brocard}. A construction for the Brocard porism is given in \cite[Theorem 4.20, p. 129]{akopyan2007-conics}.

\subsection*{Structure of the Paper}

In Section~\ref{sec:homothetic} (resp. Section~\ref{sec:brocard-porism}) we describe 3-periodic invariants associated with the homothetic pair (resp. Brocard porism). In Section~\ref{sec:similarity} we prove that both families are images of one another under a variable similarity transform. Section~\ref{sec:conclusion} presents a summary of our findings as well as a few open questions and a list of videos of some phenomena covered herein.

Appendices are included which provide (i) explicit expressions for Brocard porism vertices, (ii) a construction for the complete (2d) family of triangles with fixed Brocard points, and (iii) explicit expressions for the locus of key triangle centers in either family. 

\section{Homothetic Poncelet}
\label{sec:homothetic}
Note: when referring to triangle centers below, we will adopt the $X_k$ notation after \cite{etc}, e.g., $X_2$ is the barycenter, $X_3$ the circumcenter, etc.

Consider the family of Poncelet 3-periodics in a {\em homothetic pair}, i.e., inscribed in an ellipse $(a,b)$ and circumscribed about a concentric, axis-aligned, half-sized ellipse with semi-axes $(a',b')=(a/2,b/2)$. Notice this pair satisfies a condition for the existence of a 3-periodic Poncelet family in a concentric, axis-aligned ellipse pair, namely \cite{georgiev2012-poncelet}:

\begin{equation}
    \frac{a'}{a}+\frac{b'}{b}=1
    \label{eqn:pair-n3}
\end{equation}

Referring to Figure~\ref{fig:first-brocard-locus}:

\begin{lemma}
The family of 3-periodics in the homothetic pair conserves both area and sum of sidelengths squared.
\end{lemma}

\begin{figure}
    \centering
\includegraphics[width=\textwidth]{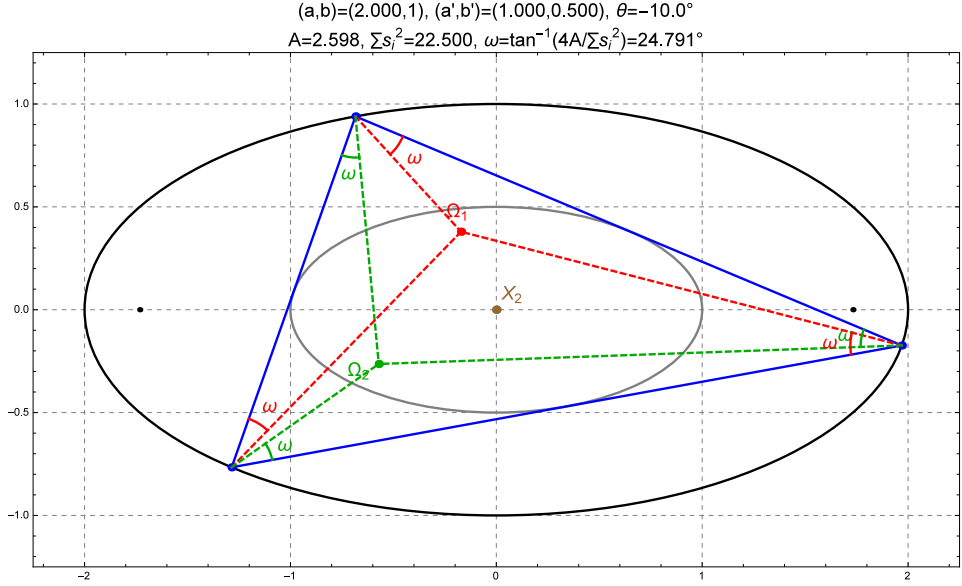}
    \caption{Area and sum of squared sidelengths of 3-periodics (blue) in the homothetic Poncelet pair (barycenter $X_2$ stationary at the origin) are invariant, and therefore so is the Brocard angle $\omega$. \href{https://youtu.be/2fvGd8wioZY}{Video}}
    \label{fig:first-brocard-locus}
\end{figure}

The proof below was kindly contributed by S. Tabachnikov \cite{sergei2020-private-sidelengths}.

\begin{proof}
Area conservation stems from the fact that the family is the affine image of 3-periodics in a concentric circular Poncelet pair. Invariant sum of squared sidelengths follows from the fact that the average of the harmonics of degree 1 and 2 over the group of rotations of order 3 is zero. Namely, consider a unit vector $v(\alpha)=(\cos \alpha, \sin \alpha)$ and a matrix $\mathcal{A}$ taking concentric circles to homothetic ellipses. Then $|\mathcal{A}v(\alpha)|^2$ is a trigonometric polynomial of degree 2. Average it over $\mathbb{Z}_3$ by adding $2\pi/3$ and $4\pi/3$ to $\alpha$. The result is independent of $\alpha$, as needed.
\end{proof}

\begin{remark}
The invariant area and sum of squared sidelengths of 3-periodics in the homothetic pair are given by:
\[ A=\frac{3\sqrt{3}ab}{2}, \;\;\;\;\; \sum{s_i}^2=\frac{9}{2}(a^2+b^2)\cdot\]
\end{remark}

\begin{theorem}
3-periodics in the homothetic pair are equibrocardal, i.e., they conserve Brocard angle $\omega$ given by: 
\[\cot\omega=\frac{\sqrt{3}(a^2+b^2)}{2ab}.\]
\label{thm:sum-cots}
\end{theorem}

\begin{proof}
A known relation is $\cot(\omega)=\sum(s_i)^2/(4A)$ \cite[Brocard Angle]{mw}. Since both numerator and denominator are conserved the result follows.
\end{proof}

Surprisingly, the sum of cotangents is invariant for $N$-periodics in the concentric, homothetic pair for any $N$ \cite{reznik2020-percolating}.

Theorem~\ref{thm:sum-cots} is also a direct consequence to a beautiful Theorem by Johnson \cite[Theorem 487, p. 291]{johnson1960}, namely that if two equilateral triangles in an oblique plane are projected onto a horizontal one, the projected triangles will have the same Brocard angle. Johnson's projection can be regarded as the affine transformation that takes Poncelet equilaterals interscribed between two concentric circles into the homothetic pair. Let $\varphi$ denote the angle between the oblique planes and the horizontal. Johnson derives \cite[p. 292]{johnson1960}:

\[ \cot\omega = \frac{\sqrt{3}}{2}\left(\cos\varphi+\frac{1}{\cos\varphi}\right) \]

We show elsewhere \cite{garcia2020-brocard-loci} that over the homothetic family, the loci of the Brocard points are two ellipses rotated an equal in opposite directions about the $x$-axis. Furthermore, these ellipses are concentric and similar to the Poncelet pair ellipses.

\section{Brocard Porism}
\label{sec:brocard-porism}
A remarkable porism is known \cite{bradley2011-brocard,bradley2007-brocard,shail1996-brocard} for a family of triangles circumscribed in a circle and circumscribing the so-called {\em Brocard inellipse} \cite{bradley2007-brocard}. Remarkably, the family is {\em equibrocardal} (conserves $\omega$) (a term from \cite{johnson1960}) {\em and} their Brocard points are stationary at the foci of the inellipse.

Given a triangle, the outer conic is the circumcircle. The inner one is the one centered on $X_{39}$ with Brianchon point is $X_6$, i.e., the touchpoints are the intersections of the cevians through $X_6$ (the symmedians) with the sidelines \cite[Brocard Inellipse]{mw}.

Assume the inellipse is centered at $(0,0)$ and its semi-axes are $(a,b)$. 

\begin{proposition}
The circumcenter $X_3$, circumradius $R$, and Brocard angle $\omega$ of the Brocard porism are given by:

\begin{equation}
X_3=[0,-\frac{c\delta_1}{b}],\;\;\;R= \frac{2a^2}{b}, \;\;\; \cot\omega=  \frac{\delta_1}{b} \\
 \label{eqn:broc-circumcircle}
\end{equation}
where $\delta_1=\sqrt{4a^2-b^2}$.
\label{prop:cotw}
\end{proposition}

Derivable from the above is a known requirement for the Brocard porism to be possible \cite[Eqs. 15--17]{shail1996-brocard}:

\begin{remark}
$R{\geq}2c$
\label{rem:minR}
\end{remark}

\begin{corollary}
\[ c = R\sin{\omega}\sqrt{1-4\sin^2{\omega}} \]
\end{corollary}

This stems from a beautiful relation presented by Shail \cite{shail1996-brocard} whereby the distance between the Brocard points is given by:

\[
|{\Omega_1}-{\Omega_2}|^2=4 R^2\sin^2{\omega}(1-4\sin^2{\omega})
\]

It be shown that the above reduces to $4 c^2$.

Explicit expressions for the Brocard vertices appear in Appendix~\ref{app:brocard-vertices}. A few configurations of the family are depicted in Figure~\ref{fig:brocard-parametric}. 

\begin{figure}
    \centering
    \includegraphics[width=\textwidth]{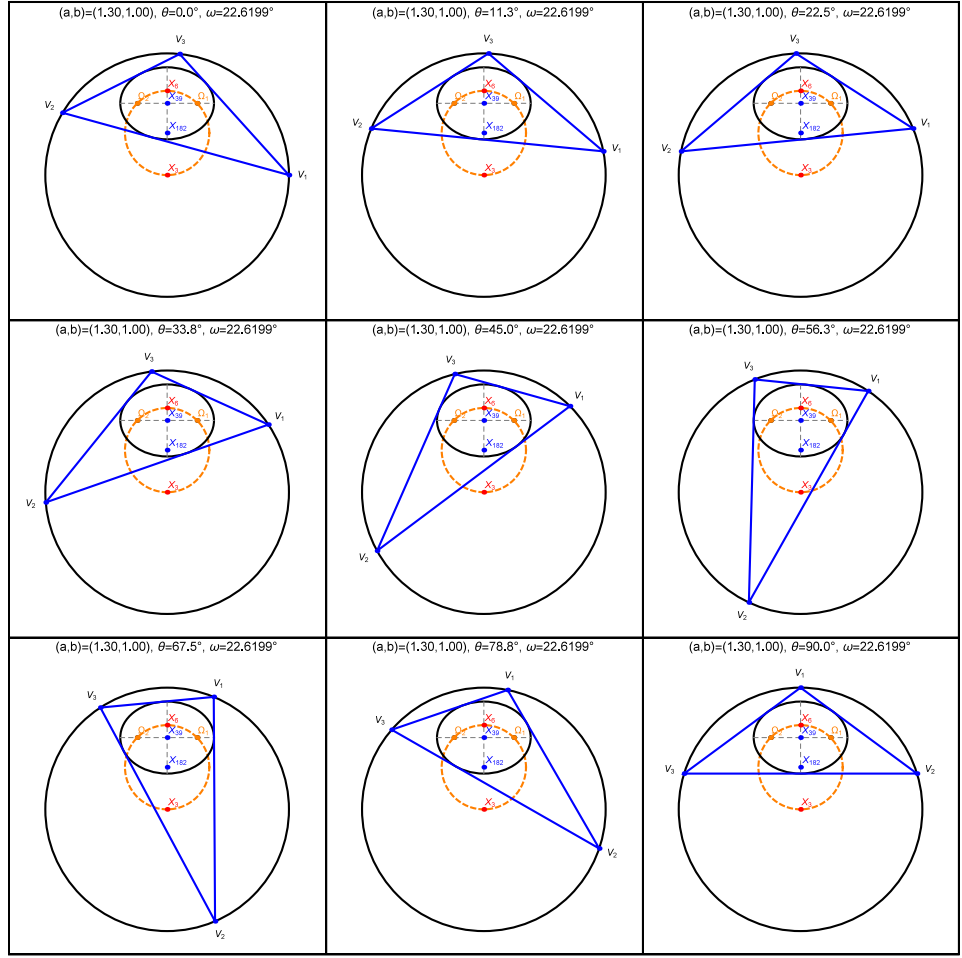}
    \caption{Nine positions of the 3-periodic Poncelet family inscribed in a circle (black) and circumscribing an ellipse (black). All triangles have the same Brocard angle $\omega$ and stationary Brocard point $\Omega_1$ and $\Omega_2$ which lie at the foci of the (Brocard) inellipse. Their midpoint $X_{39}$ is at the center of the inellipse. Triangle centers $X_3$ and $X_6$ are stationary and are concyclic with the Brocard points on the Brocard circle (dashed orange) whose center is $X_{182}$.}
    \label{fig:brocard-parametric}
\end{figure}


\subsection{All triangles with fixed Brocard Points}

As seen above, the Brocard porism gives rise to a 1d family of triangles whose Brocard points coincide with the foci of the caustic, i.e., the Brocard inellipse.

Consider the 1-parameter family of origin-centered confocal ellipses $\mathcal{E}_\lambda$:

\[\mathcal{E}_\lambda:\;\frac{x^2}{a^2-\lambda}+\frac{y^2}{b^2-\lambda}-1=0,\;\; \lambda<b^2,  \;\; 0<b<a.\]

\begin{remark}
The circumcenter $X_{3,\lambda}$, circumradius $R_\lambda$, and Brocard angle $\omega_\lambda$ for each family implied by a choice of $a,b,\lambda$ are given by: 

\begin{equation*}
     X_{3,\lambda}=\left[0,- \frac {c\, \delta_2}{\sqrt {{b}^{2}-
\lambda}}\right],\;\;\;
     R_\lambda= \frac{ 2(a^2-\lambda)}{ \sqrt{b^2-\lambda}},\;\;\;\cot\omega_\lambda=   \frac{\delta_2}{\sqrt{b^2-\lambda}} 
\end{equation*}
\end{remark}

\noindent where $\delta_2=\sqrt{4a^2-b^2-3\lambda}$.
This stems from Proposition~\ref{prop:cotw}: replace $a$ with $\sqrt{a^2-\lambda}$ and $b$ with $\sqrt{b^2-\lambda}$ and obtain expressions for t

\begin{theorem}
The family of triangles circumscribing $\mathcal{E}_\lambda$ and inscribed in a circle of radius $R_k$ centered on $X_{3,\lambda}$ has fixed Brocard points on the foci $[\pm{c},0]$ of $\mathcal{E}_\lambda$. Moreover, varying $\lambda\in(-\infty,b^2)$ covers the entire 2d family of triangles with Brocards points on $[\pm{c},0]$.
\end{theorem}

\begin{proof}
Follows from Proposition \ref{prop:cotw}.
\end{proof}

An alternative, synthetic method for constructing the same fixed-Brocard 2d family was kindly contributed by Peter Moses \cite{moses2020-private-brocard} appears in Appendix~\ref{app:moses-fixed}.

As a curiosity, Appendix~\ref{app:broc-circles} lists some circles whose centers and radii are invariant over the Brocard poristic family.


\section{Related by Similarity: Homothetic and Brocard-Poristic Families}
\label{sec:similarity}
The homothetic and Brocard-poristic Poncelet families turn out to be images of each other under a variable similarity transform. The fact that it preserves angles is consistent with both families being equibrocardal.

To prove this we only need to show that either (i) the Brocard inellipse in the homothetic pair or (ii) the Steiner circumellipse in the Brocard porism have invariant aspect ratio. In fact, both must be true.

Referring to Figure~\ref{fig:homot-broc-inell}:

\begin{lemma}
Over the homothetic pair, the Brocard inellipse has invariant semi-axes ratio $\beta$ given by
\[\beta=
\frac{\sqrt{3a^4+10a^2b^2+3b^4}}{4ab} > 1\]
\label{lem:beta}
\end{lemma}

\begin{proof}
The semi-axes of the Brocard inellipse of a triangle $ABC$ are given explicitly in terms of sidelengths $a,b,c$ in \cite[Brocard Inellipse, Eqns. 3,4]{mw}. Combining it with the expression for $\sin\omega$ in \cite[Eqn. 7]{mw}, obtain:

\[ \frac{4\Delta}{\sqrt{\Gamma}}=2\sin\omega,\;\; \Gamma=a^2b^2+a^2c^2+b^2c^2.\]

Where $\Delta$ is the triangle area. Applying this result to the vertices of 3-periodics in the homothetic pair (see Appendix~\ref{app:brocard-vertices}), the result follows.
\end{proof}

\begin{figure}
    \centering
    \includegraphics[width=\textwidth]{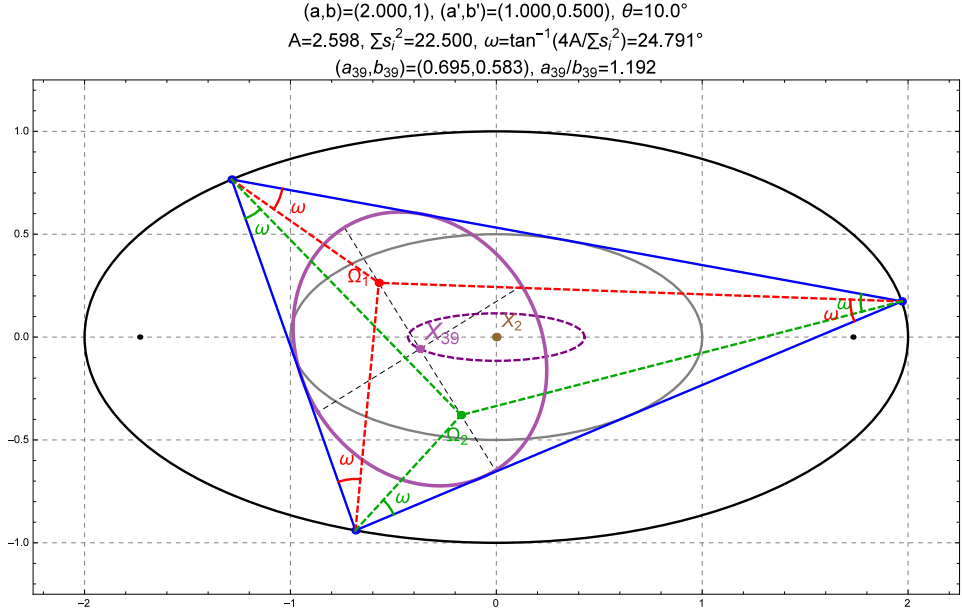}
    \caption{The center of the Brocard inellipse (purple) is $X_{39}$. Over the homothetic family, its locus (dashed purple) is an ellipse concentric and axis-aligned with the pair (see Appendix~\ref{app:x2-x39}). Over the family the axes of the Brocard inellipse are variable, however its aspect ratio is constant. Therefore this family can be regarded as the image of the Brocard-poristic family under a variable similarity transform. \href{https://youtu.be/DIm2qTxGWXE}{Video}}
    \label{fig:homot-broc-inell}
\end{figure}

Referring to Figure~\ref{fig:broc-steiner}:

\begin{lemma}
Over the Brocard porism, the Steiner circumellipse has invariant semi-axes ratio $\sigma$ given by
 
\[\sigma^2=\frac{8a^2-5b^2+4\sqrt{4a^4-5a^2b^2+b^4}}{3b^2}  \]
 
\label{lem:sigma}
\end{lemma}

\begin{proof} 
The Steiner circumellipse $\mathcal{S}(t)$ is centered in $X_2$ and passes through the vertices $P_1(t), P_2(t),P_3(t)$ (see Appendix~\ref{app:brocard-vertices}. With a computer-aided algebra system (CAS) we obtain the following implicit equation for it:

\[ \mathcal{S}(x,y)=a_{20}(t) x^2+2 a_{11}(t) {x}{y}+a_{02}(t)y^2+a_{10}(t) x+a_{01}(t)y+a_{00}(t)=0.\]

The ratio of its semi-axes is obtained from its Hessian matrix $H$ as follows \cite{strang2006-linear}: 

\[ \sigma(t)=\frac{ \mathrm{tr}\hspace{1pt}(H)+\sqrt{\mathrm{tr}\hspace{1pt}(H)^2-4\det(H)}}{2\,\det(H)},\]

Using a CAS, verify $\sigma$ is independent of $t$ and equal to the expression in the claim.
\end{proof}

\begin{figure}
    \centering
    \includegraphics[width=\textwidth]{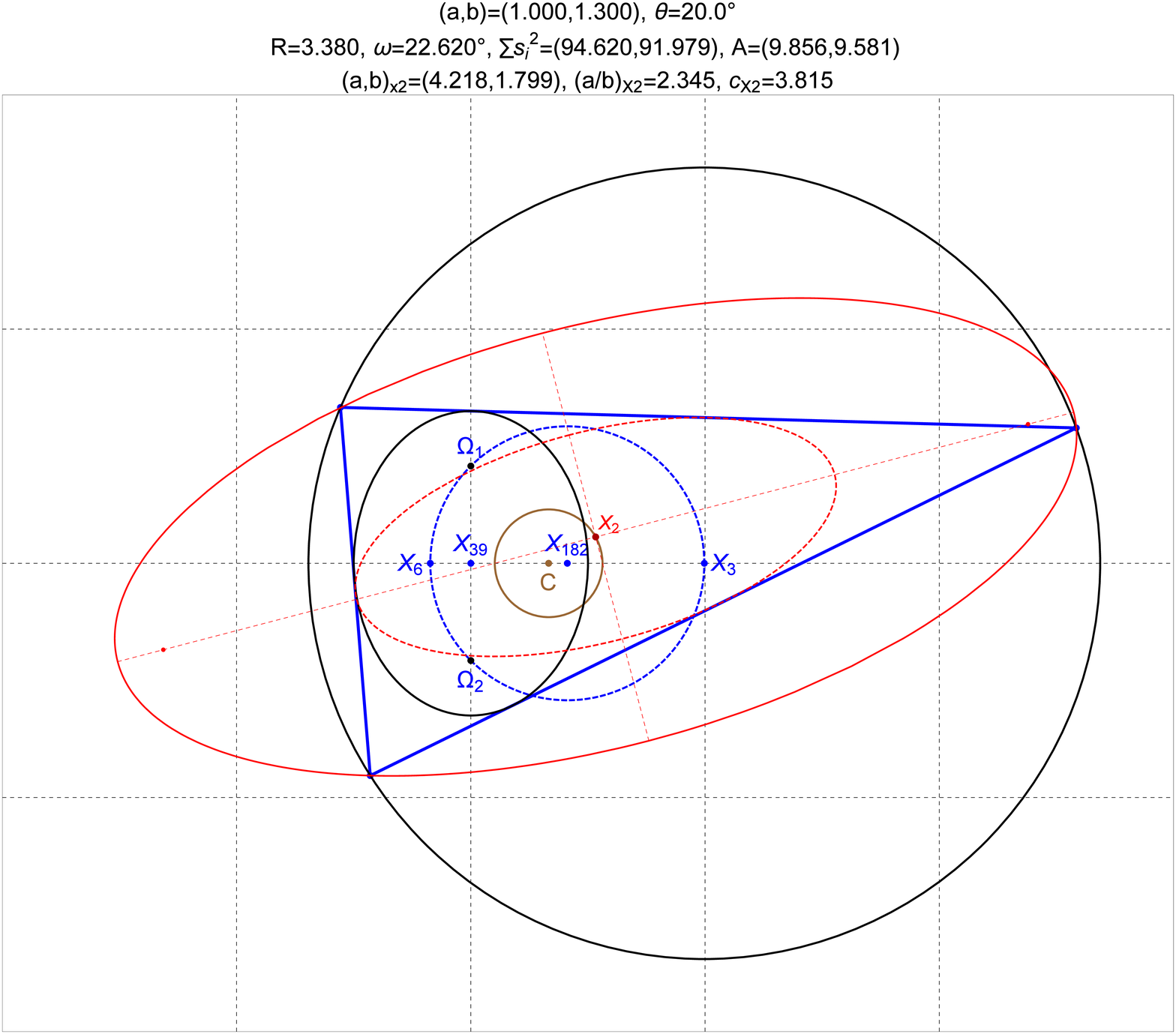}
    \caption{The Steiner Circumellipse (red) and Inellipse (dashed red) for a 3-periodic in the Brocard porism (blue) is shown. Both ellipses are by definition centered on $X_2$ whose locus is a circle centered on $C$ (see Appendix~\ref{app:x2-x39}. Over the family, the lengths of their axes are variable, however their ratio is invariant. Therefore Brocard-poristic 3-periodics can be regarded as the image under a variable similarity transform of the homothetic family. Note this is consistent with both systems being equibrocardal. \href{https://youtu.be/h3GZz7pcJp0}{Video}}
    \label{fig:broc-steiner}
\end{figure}

\begin{theorem}
The 3-periodic family in a homothetic pair is similar to 3-periodics arising from a 1d family of Brocard porisms. Conversely, the 3-periodic family in a Brocard porism pair is similar to 3-periodics arising in a 1d family of homothetic pairs.
\end{theorem}

\begin{proof}
Consider the transformation 
\[X'=Scale(k/b_{39}).Rot(-\theta).Transl(-X_{39}).X\].

\noindent where $k$ is a chosen constant, $b_{39}$ is the variable minor semi-axis length of the the (moving) Brocard inellipse in the homothetic pair, $\theta$ the angle between said minor axis and the horizontal, and $X_{39}$ is the moving center of the inellipse. Clearly, the transformation will take the moving Brocard inellipse to an origin-centered, upright one. By Lemma~\ref{lem:beta}, the ratio $\beta$ of inellipse semi-axes $a_{39}/b_{39}$ is constant, implying the transformed inellipse will have fixed axes $(k\beta, k)$. Notice its circumcenter and circumradius are prescribed by the semi-axes of the caustic (Equation~\ref{eqn:broc-circumcircle}). So the homothetic family will be mapped to a 1-parameter family of Brocard porisms where the parameter is $k$. 
\end{proof}

The reverse argument can be used to transform the moving Steiner circumellipse in the Brocard porism to an origin centered, stationary one with axes $(k'\sigma,k')$, namely:

\[X=Scale(k'/b_{2}').Rot(-\theta').Transl(-X_{2}').X'\]

Where primed (unprimed) quantities refer to those measured in the Brocard porism (resp. homothetic pair). See Figures~\ref{fig:homot-broc-inell} and \ref{fig:broc-steiner} and the videos mentioned in the captions.


	

\section{Conclusion}
\label{sec:conclusion}
Previously we had shown the similar, concentric doppelgänger to the poristic triangle family (fixed incircle and circumrcircle) were 3-periodics in the confocal pair (elliptic billiard) \cite{garcia2020-poristic}, and that both conserved the sum of cosines. Later we proved the sum of cosines is conserved by the confocal pair for all $N$ \cite{reznik2020-intelligencer,akopyan2020-invariants}.

Here we make a similar argument: the concentric doppelgänger to 3-periodics in the non-concentric Brocard porism is the concentric, axis, aligned homothetic family. Both conserve the Brocard angle and therefore the sum of cotangents. Suprisingly, homothetic N-periodics conserve the sum of cotangents for all $N$ \cite{reznik2020-percolating}.

These considerations are summarized on Table~\ref{tab:concentric-duals}.

The following are questions:

\begin{itemize}
    \item are there other (non-concentric,concentric) duos related by similarity? What common quantity do they conserve?
    \item must the outer conic of the non-concentric pair be a circle so a concentric, similar doppelgänger can be found?
    \item How does this similar duality relate to transformations mentioned in \cite{dragovic11} which, under a suitable ambient ($CP^2$) a transformation exists which takes any non-concentric, non-axis-aligned conic pair into a confocal, canonical one?
\end{itemize}

\begin{table}
\centering
\begin{tabular}{|c|c|c|}
\hline
Concentric &
\makecell[c]{Common\\Invariants} & Non-Concentric \\
        \hline
        \makecell[t]{\textbf{Confocal} (Billiard)
        \\$a_9$, $b_9$, $X_9$, $\sum{s_i}$} &
        \makecell[t]{$r/R$\\ $\sum\cos{\theta_i}$} & \makecell[t]{\textbf{Chappple Porism}
        \\
        $r$, $R$, $X_1,X_3,\ldots$\\$a_9/b_9$} \\
        \hline
        \makecell[t]{\textbf{Homothetic} \\
        $a_2$, $b_2$, $X_2$, $\sum{s_i}^2$\\$a_{39}$,$b_{39}$} &
        \makecell[t]{$\omega$\\ $\sum\cot{\theta_i}$} & \makecell[t]{\textbf{Brocard Porism}\\
        $a_{39}$, $b_{39}$, $X_3,X_6,X_{39},X_{182},\ldots$\\
        $a_2/b_2$}\\
        \hline
    \end{tabular}
    \caption{\textbf{Left column}: invariants held in two concentric systems: confocal and homothetic. \textbf{Right column}: invariants held by their non-concentric, similar counterparts: Chapple- and Brocard-poristic. \textbf{Middle column}: invariants which hold on both the concentric and non-concentric counterpart. The confocal-Chapple similarity was studied in \cite{garcia2020-poristic}.}
    \label{tab:concentric-duals}
\end{table}

A table of animations of some of the results above appears below.

\begin{table}[H]
\begin{tabular}{|c|c|l|}
\hline
Exp & Video & Title \\
\hline
01 & \href{https://youtu.be/8hkeksAsx0E}{*} & Family of 3-Periodics in Five Poncelet Pairs \\
02 & 
\href{https://youtu.be/2fvGd8wioZY}{*} &  
\makecell[lt]{Poncelet 3-periodics in the Homothetic Pair\\conserve Brocard angle} \\
03 & \href{https://youtu.be/JANPPLET0so}{*}  & \makecell[lt]{Brocard-Poncelet Porism with stationary\\Brocard Points and invariant Brocard Angle} \\
04 & \href{https://youtu.be/h3GZz7pcJp0}{*} &  \makecell[lt]{Brocard Porism and Invariant Aspect Ratio\\ Steiner Circumellipse} \\
05 & \href{https://youtu.be/DIm2qTxGWXE}{*} &  \makecell[lt]{Homothetic Pair and Invariant Aspect Ratio\\ Brocard Inellipse} \\

\hline
\end{tabular}
\caption{Experimental animations. Click on the * to see it as a {YouTube} video and/or a browser-based simulation.}
\label{tab:videos}
\end{table}

\noindent We would like to thank A. Akopyan who originally challenged us to find an equibrocardal Poncelet pair. S. Tabachnikov  contributed the proof for invariant squared sidelengths on the homothetic pair applicable to all $N$. P. Moses contributed several insights, as well as the geometric construction for the family with fixed Brocard points as well as suggesting relevant references. L. Gheorghe pointed us to a relevant set of porism results by Bradley and provided valuable feedback on our videos. M. Helman has kindly assisted us with simulations and several insights. The first author is fellow of CNPq and coordinator of Project PRONEX/ CNPq/ FAPEG 2017 10 26 7000 508.

\appendix
\section{Brocard Porism: Vertices of 3-periodics}
\label{app:brocard-vertices}
Let the triangle vertices be $P_i(t),i=1,2,3$. Parametrize $P_1(t)=X_3+R[\cos{t},\sin(t)]$. 

\begin{align*}
    P_1(t)=& [ {\frac {2{a}^{2}}{b}}\cos t ,-{\frac {c\delta_1}{b}}+
 {\frac {2{a}^{2}\sin t }{b}}]\\
P_2(t)=&[\frac{p_{2,x}}{q}, \frac{p_{2,y}}{q}]\\
P_3(t)=&[\frac{p_{3,x}}{q}, \frac{p_{3,y}}{q}]\\
p_{2,x}=&-4\,{a}^{2}b{c}^{2} \left( 4\,{a}^{4}-3\,{a}^{2}{b}^{2}+{b}^{4}
 \right)  \cos^3t-8\,{a}^{4}b{c}^{3
}\delta_1\,\sin t  \cos^3 t \\ -&2\,\Delta\,{a}^{2}c\delta_1\, \left( 2\,{a}^{4}-2\,a^{2}b^{2}+b^{4} \right)  \cos^2t 
- 4
\,\Delta\,{a}^{4}{c}^{2} \left( 2\,{a}^{2}-{b}^{2} \right) \sin
 t  \cos^2t\\
 +&{a}^{2}b^{3} \left( 4\,{a}^{4}-7\,{a}^{2}{b}^{2}+2\,{b}^{4} \right) \cos t  
 + 2\,{a}^{2}{b}^{3}{c}^{3}\delta_1\,\sin t \cos t \\
 +&{a}^{4}\Delta\,{b}^{2} \left( 2\,{a}^{
2}-{b}^{2} \right) \sin t +{a}^{4}{b}^{2}c
\delta_1\Delta
\\
p_{2,y}=&8\,{a}^{4}b{c}^{3}\delta_1\, \cos^4t
+4\,\Delta\,{a}^{4}{b}^{2}{c}^{2}  \cos^3 t 
 -2\,{c}^{2}\delta_1\,{b}^{3}c \left( 5\,{a}^{2}-2\,{b}^{2}
 \right)  \cos^2t\\
 -&4\,{b}^{3}{a}^{2}
{c}^{2} \left( 3\,{a}^{2}-{b}^{2} \right) \sin t 
 \cos^2t-{b}^{2}\Delta\,{a}^{2}
 \left( 8\,{a}^{4}-9\,{a}^{2}{b}^{2}+2\,{b}^{4} \right) \cos t\\
 -&2\,{b}^{2}\Delta\,{a}^{2}c\delta_1\, \left( 2\,{a}^{2}-{b}^{2
} \right) \sin t \cos t -\sin \left( t
 \right) {a}^{4}{b}^{5}
\\
p_{3,x}=&-4\,{a}^{2}b{c}^{2} \left( 4\,{a}^{4}-3\,{a}^{2}{b}^{2}+{b}^{4}
 \right)  \cos^3t-8\,{a}^{4}b{c}^{3
}\delta_1\,\sin t   \cos^3 t \\
 -&2\,\Delta\,{a}^{2}c\delta_1\, \left( 2\,{a}^{4}-2\,{a}^{2
}{b}^{2}+{b}^{4} \right)  \cos^2t\\
-&4
\,\Delta\,a^{4}c^{2} \left( 2\,{a}^{2}-{b}^{2} \right) \sin
 t  \cos^2t\\
 +&a^{2}b^{3} \left( 4\,{a}^{4}-7\,{a}^{2}{b}^{2}+2\,{b}^{4} \right) \cos t  
 + 2\,{a}^{2}{b}^{3}{c}^{3}\delta_1\,\sin t \cos t \\
 +&{a}^{4}\Delta\,{b}^{2} \left( 2\,{a}^{
2}-{b}^{2} \right) \sin t +\Delta\,{a}^{4}{b}^{2}c
\delta_1
\\
p_{3,y}=&8\,{a}^{4}b{c}^{3}\delta_1\, \cos^4t
+4\,\Delta\,{a}^{4}{b}^{2}{c}^{2}  \cos^3 t 
 -2\,{c}^{3}\delta_1\,{b}^{3} \left( 5\,{a}^{2}-2\,{b}^{2}
 \right)  \cos^2t\\
 -&4\,a^2{b}^{3} 
{c}^{2} \left( 3\,{a}^{2}-{b}^{2} \right) \sin t 
 \cos^2t-{a}^{2}{b}^{2} \Delta\left( 4\,
 {a}^{2}c\delta_1 
- 2\, {b}^{2}c\delta_1 \right) \sin
 t \cos t\\
 -&{a}^{2}{b}^{2} \Delta \left( 8\,
 \,{a}^{4}-9\, \,a^{2}b^{2}+2\, {b}^{4} \right) 
\cos t -\sin t a^{4}b^{5}
\\
q=& 16\,a^{4}c^{4} \cos^4t-4\,b^{
2}c^{2} \left( 2\,{a}^{4}-{b}^{2}\delta_1^{2} \right)  
  \cos^2 t  +{a}^{4}{b}^{4}
\\
\end{align*}





\section{Peter Moses' Construction for Fixed Brocard Families}
\label{app:moses-fixed}
The following construction for the family of triangles with fixed Brocard Points was kindly contributed by Peter Moses \cite{moses2020-private-brocard}.

Given fixed Brocard Points $\Omega_1$ and $\Omega_2$ and vertex $A$ of triangle $ABC$, compute $B,C$ as follows:

\begin{itemize}
\item $O_a$ is the circumcenter of $A{\Omega_1}{\Omega_2}$.
\item $A'$ is the reflection of $A$ in the midpoint of ${\Omega_1}{\Omega_2}$.
\item The perpendicular through $A'$ to $A{O_a}$ intersects $A{\Omega_1}$ at $P$, $A{\Omega_2}$ at $Q$ and $\Omega_1{\Omega_2}$ at $R$.
\item $A''$ is the reflection of $A$ in the midpoint of $Q{\Omega_1}$.
\item $A'''$ is the reflection of $A$ in the midpoint of $P{\Omega_2}$.
\item Intersect the circle through $A''$ centered at $Q$ with the conic ${A}{\Omega_1}{R}{Q}{A''}$ at $A''$ and $C$.
\item Intersect the circle through $A'''$ centered at $P$ with the conic ${A}{\Omega_2}{R}{P}{A'''}$ at $A'''$ and $B$.
\end{itemize}

Claim: $ABC$ has Brocard points $\Omega_1$ and $\Omega_2$.


\section{Loci of Homothetic X(39)  Brocard-poristic X(2)}
\label{app:x2-x39}
\subsection{Homothetic pair: elliptic locus of X(39)}

Recall the center of the Brocard inellipse is $X_{39}$ \cite[Brocard Inellipse]{mw}.

\begin{proposition}
Over the homothetic pair, the locus of $X_{39}$ is an ellipse concentric and axis-aligned with the pair with semi-axes:

\begin{equation*}
(a_{39},b_{39}) = \frac{c^2}{2}\left(\frac{a}{a^2 + 3 b^2}, \frac{b}{3 a^2 + b^2}\right)
\end{equation*}
\end{proposition}
\begin{proof} Consider the triangular orbit of the homothetic pair \[A=[a\cos t,b\sin t], B=[a\cos(t+\frac{2\pi}{3}),b\sin(t+\frac{2\pi}{3}) ], C=[a\cos(t+\frac{4\pi}{3}),b\sin(t+\frac{4\pi}{3}) ].\]
Using the trilinear coordinates of $X_{39}$ given by $[p:q:r]=[a(b^2+c^2): b(a^2+c^2): c(a^2+b^2)]$ it follows that
\[X_{39}(t)=\frac{p a A+qbB+rcC}{ap+bq+cr}, \;\; a=|B-C|, b=|A-C|, c=|A-B|.\]
Therefore, using symbolic calculations, it follows that
\[X_{39}(t)=\left[ -\frac{a(a^2-b^2)\cos(3t)}{ 2 (a^2+3b^2)},-\frac{ b(a^2-b^2)\sin(3t)}{ 2(3a^2+b^2)}\right]\cdot  \]
This ends the proof.
\end{proof}

\subsection{Brocard porism: circular locus of X(2)}

Recall the center of the Steiner circumellipse is $X_{2}$ [Steiner Circumellipse]\cite{mw}. The following was contributed by Peter Moses:

\begin{proposition}
Over the Brocard porism, the locus of $X_2$ is a circle centered on triangle center $X_{11171}$ and of radius $R(2\cos(2\omega)-1)/3$.
\end{proposition}

Note: Trilinear coordinates for $X_{11171}$ are $2\cos{A}-\cos(A + 2\omega) ::$ \cite[Part 6]{etc}. 

In addition, the following expressions for the $X_2$ locus center and radius have been derived in terms of $a$ and $b$:

\[ C_2(x,y)={x}^{2
}+{y}^{2}+  {\frac {2\sqrt {4\,{a}^{4}-5\,{a}^{2}{b}^{2}+{b}^{4}} }{3b}}y+\frac{1}{3}({a}^{2}- {b}^{2})=0\]
centered at $[0, - \,{\frac {\sqrt {4\,{a}^{4}-5\,{a}^{2}{b}^{2}+{b}^{4}}}{3b}}]$
and radius $2(a^2-b^2)/(3b)$.


The circular locus of $X_2$ can also be inferred from results in \cite{bradley2007-brocard}.

We leave it as an exercise:

\begin{remark}
The axes of the Steiner circumellipse and the circular locus of $X_2$ intersect the minor axis of the Brocard inellipse on the same locations.
\end{remark}

\section{Stationary Circles over the Brocard Porism}
\label{app:broc-circles}
Table~\ref{tab:stationary-circles} lists circles named in \cite{mw} which over the Brocard porism have stationary centers (indicated) and invariant radii (since they only depend on $R$ and $\omega$).

The coordinates for $X_{182}$ and radius $R_{182}$ of the Brocard circle in terms of $a,b$ are given by:

\[ X_{182}=[0, -c\frac{\sqrt{2a^2-b^2}}{b\sqrt{4a^2-b^2}}],\;\;\;R_{182}= \frac{2a^2c}{b\sqrt{4a^2-b^2}}\cdot \]

\renewcommand{\arraystretch}{0.95}
\begin{table}
\small
\begin{tabular}{|r|c|l|}
\hline
Circle & Center & Radius \\
\hline
Circumcircle & $X_3$ & $R$ \\
2nd Brocard & $X_3$ & $e R$ \\
Stammler & $X_3$ & $2 R$ \\
2nd Lemoine (Cosine$^\dagger$) & $X_6$ & $R\tan\omega$ \\
Brocard & $X_{39}$ & $e (R/2) \cos\omega$ \\
Gallatly & $X_{39}$ & $R\sin{\omega}$ \\
Half-Moses & $X_{39}$ & ${R}\sin^2\omega$ \\
Moses & $X_{39}$ & $2{R}\sin^2\omega$
\\
1st Lemoine & $X_{182}$ & $(R/2) \sec\omega$ \\
Lucas Inner & $X_{6407}$ & $R/(4\cot\omega+7)$ \\
\hline
\end{tabular}
\caption{Circles named in \cite{mw} which remain stationary over the Brocard porism. Let $e = \sqrt{1-4\sin^2{\omega}}$. $^\dagger$The Cosine Circle to the excentrals of 3-periodics in the elliptic billiard (confocal pair) is also stationary \cite{garcia2020-new-properties}.}
\label{tab:stationary-circles}
\end{table}



\bibliographystyle{maa}
\bibliography{references,authors_rgk}

\begin{thebibliography}{10}
\expandafter\ifx\csname urlstyle\endcsname\relax
 \providecommand{\url}[1]{doi:\discretionary{}{}{}#1}\else
 \providecommand{\url}{doi:\discretionary{}{}{}\begingroup
  \urlstyle{rm}\Url}\fi

\bibitem{akopyan2020-invariants}
Akopyan, A., Schwartz, R., Tabachnikov, S. (2020).
\newblock Billiards in ellipses revisited.
\newblock \emph{European Journal of Mathematics}.

\bibitem{akopyan2007-conics}
Akopyan, A.~V., Zaslavsky, A.~A. (2007).
\newblock \emph{Geometry of Conics}.
\newblock Providence, RI: Amer. Math. Soc.

\bibitem{bradley2011-brocard}
Bradley, C. (2011).
\newblock The geometry of the brocard axis and associated conics.
\newblock \url{people.bath.ac.uk/masgcs/Article116.pdf}.
\newblock CJB/2011/170.

\bibitem{bradley2007-brocard}
Bradley, C., Smith, G. (2007).
\newblock On a construction of {H}agge.
\newblock \emph{Forum Geometricorum}, 7: 231--–247.
\newblock \url{forumgeom.fau.edu/FG2007volume7/FG200730.pdf}.

\bibitem{coolidge1971}
Coolidge, J.~L. (1971).
\newblock \emph{A Treatise on the Geometry of the Circle and Sphere}.
\newblock New York,NY: Chelsea.

\bibitem{dragovic11}
Dragovi\'{c}, V., Radnovi\'{c}, M. (2011).
\newblock \emph{Poncelet Porisms and Beyond: Integrable Billiards,
  Hyperelliptic Jacobians and Pencils of Quadrics}.
\newblock Frontiers in Mathematics. Basel: Springer.
\newblock \url{books.google.com.br/books?id=QcOmDAEACAAJ}.

\bibitem{gallatly1914-geometry}
Gallatly, W. (1914).
\newblock \emph{The modern geometry of the triangle}.
\newblock Francis Hodgson.

\bibitem{garcia2020-brocard-loci}
Garcia, R., Reznik, D. (2020).
\newblock Loci of the brocard points over some triangle families.
\newblock In preparation.

\bibitem{garcia2020-poristic}
Garcia, R., Reznik, D. (2020).
\newblock Similar families: Poristic triangles and 3-periodics in the elliptic
  billiard.
\newblock \url{arxiv.org/abs/2004.13509}.

\bibitem{garcia2020-new-properties}
Garcia, R., Reznik, D., Koiller, J. (2020).
\newblock New properties of triangular orbits in elliptic billiards.
\newblock \url{arxiv.org/abs/2001.08054}.

\bibitem{georgiev2012-poncelet}
Georgiev, V., Nedyalkova, V. (2012).
\newblock Poncelet’s porism and periodic triangles in ellipse.
\newblock \emph{Dynamat}.
\newblock \url{www.dynamat.oriw.eu/upload_pdf/20121022_153833__0.pdf}.

\bibitem{honsberger1995}
Honsberger, R. (1995).
\newblock \emph{Episodes in Nineteenth and Twentieth Century Euclidean
  Geometry}.
\newblock Mathematical Association of America.

\bibitem{johnson1960}
Johnson, R.~A. (1960).
\newblock \emph{Advanced Euclidean Geometry}.
\newblock New York, NY: Dover, 2nd ed.
\newblock \url{bit.ly/33cbrvd}.
\newblock Editor John W. Young.

\bibitem{etc}
Kimberling, C. (2019).
\newblock Encyclopedia of triangle centers.
\newblock \url{faculty.evansville.edu/ck6/encyclopedia/ETC.html}.

\bibitem{moses2020-private-brocard}
Moses, P. (2020).
\newblock Family of triangles with fixed brocard points.
\newblock Private Communication.

\bibitem{pamfilos2004}
Pamfilos, P. (2004).
\newblock On some actions of {$D_3$} on a triangle.
\newblock \emph{Forum Geometricorum}, 4: 157--176.
\newblock \url{http://forumgeom.fau.edu/FG2004volume4/FG200420.pdf}.

\bibitem{reznik2020-percolating}
Reznik, D., Garcia, R., Galkin, S. (2020).
\newblock Percolating invariants of poncelet n-periodics in concentric pairs.
\newblock In preparation.

\bibitem{reznik2020-intelligencer}
Reznik, D., Garcia, R., Koiller, J. (2019).
\newblock Can the elliptic billiard still surprise us?
\newblock \emph{Math Intelligencer}, 42.
\newblock \url{rdcu.be/b2cg1}.

\bibitem{shail1996-brocard}
Shail, R. (1996).
\newblock Some properties of brocard points.
\newblock \emph{The Mathematical Gazette}, 80(489): 485–491.

\bibitem{strang2006-linear}
Strang, G. (2006).
\newblock \emph{Linear Algebra and Its Application, 4th Edition}.
\newblock Belmont, CA: Thomson, Brooks/Cole.

\bibitem{sergei2020-private-sidelengths}
Tabachnikov, S. (2020).
\newblock Invariant sum of squared sidelengths for {N}-periodics in the
  homothetic poncelet pair.
\newblock Private Communication.

\bibitem{mw}
Weisstein, E. (2019).
\newblock Mathworld.
\newblock \emph{MathWorld--A Wolfram Web Resource}.
\newblock \url{mathworld.wolfram.com}.

\end{thebibliography}

\end{document}